\documentclass[reqno,12pt,reqno]{amsart}
\usepackage{amsfonts, amsthm, amsmath, amssymb, enumerate, verbatim, bm, cite}
\usepackage{hyperref}
\usepackage[margin=1.5in]{geometry}
\usepackage[utf8]{inputenc}
\hypersetup{colorlinks=false}

\RequirePackage{mathrsfs} \let\mathcal\mathscr

\numberwithin{equation}{section}

\newtheorem{theorem}{Theorem}[section]
\newtheorem{lemma}[theorem]{Lemma}
\newtheorem{proposition}[theorem]{Proposition}

\theoremstyle{definition}
\newtheorem*{ack}{Acknowledgements}

\newtheorem{definition}[theorem]{Definition}

\renewcommand{\rho}{\varrho}

\newcommand{\ZZ}{\mathbb{Z}}

\newcommand{\RR}{\mathbb{R}}

\renewcommand{\leq}{\leqslant}

\renewcommand{\geq}{\geqslant}

\newcommand{\bs}{\boldsymbol}

\newcommand{\x}{\mathbf{x}}
\newcommand{\y}{\mathbf{y}}

\renewcommand{\phi}{\varphi}

\newcommand{\ve}{\varepsilon}

\DeclareMathOperator{\meas}{meas}

\DeclareMathOperator{\sinc}{sinc}

\renewcommand\Im{\operatorname{Im}}
\renewcommand\Re{\operatorname{Re}}

\begin{document}
\date{\today}
\title{Some new results in quantitative Diophantine approximation}

\author{Anish Ghosh and V. Vinay Kumaraswamy}
\address{School of Mathematics, Tata Institute of Fundamental Research, Homi Bhabha Road, Colaba, Mumbai, India 400005}

\thanks{A. G. gratefully acknowledges support from a MATRICS grant from the Science and
Engineering Research Board, a grant from the Infosys foundation and a Department of
Science and Technology, Government of India, Swarnajayanti fellowship. The authors were supported by the Department of Atomic Energy, Government of India,
under project no.12-R\&D-TFR-5.01-0500}
\email{ghosh@math.tifr.res.in}
\email{vinay.visw@gmail.com} 


\begin{abstract}
In this article we establish two new results on quantitative Diophantine approximation for one-parameter families of diagonal ternary indefinite forms. In the first result, we consider quadratic forms taking values at prime points. In the second, we examine inhomogeneous forms of arbitrary degree taking values at integer points. 
\end{abstract}

\maketitle

\tableofcontents


\section{Introduction}
Let $Q(x_1,\ldots,x_s)$ be a real, nondegenerate, indefinite quadratic form in $s \geq 3$ variables, that is not a multiple of a quadratic form with rational coefficients. A famous theorem of Margulis asserts that, under these conditions,  $Q(\ZZ^s)$ is a dense subset of $\RR$. This settled an old conjecture of Oppenheim. Thus, given a real number $\ve > 0$, there exists $\x \in \ZZ^s$ such that 
\begin{equation*}
0 < |Q(\x)| < \ve.
\end{equation*} 
We may reformulate Margulis's theorem in the following manner: 
there is a sequence $\delta(N)\to 0$ (depending on $Q$) such that for all sufficiently large $N$,
\begin{equation*}
\min_{0 < |\x| < N}|Q(\x)|\leq \delta(N).
\end{equation*}
 It is then natural to seek good upper bounds for $\delta(N)$. Note that while Margulis's theorem establishes the solubility of the above inequality, it gives no upper bound on the size of $\x$ in terms of $\ve$. 

For diagonal forms in at least $5$ variables Birch and Davenport~\cite{BD58} showed using the circle method that we may take $\delta(N) \ll N^{-1/2+\ve}$. For general quadratic forms in at least $5$ variables, upper bounds for $\delta(N)$ have been obtained by Buterus, G\"{o}tze, Hille and Margulis~\cite{GotzeMargulis2010}. Using sieve methods, Iwaniec~\cite{Iwaniec77} showed that $\delta(N) \ll N^{-\ve},$ for some small $\ve >0$, for certain diagonal quadratic forms in $4$ variables; a result which has unfortunately been overlooked in subsequent mentions of Oppenheim's conjecture in the mathematical literature. 

In an important recent development, Lindenstrauss and Margulis~\cite{LM14} have shown that there exists $\kappa > 0$ such that $\delta(N) \ll (\log N)^{-\kappa},$ for an explicit set of ternary quadratic forms satisfying a Diophantine condition, and this is currently the state of the art for ternary quadratic forms.

However, one can do significantly better by averaging over suitable families of quadratic forms. Indeed, let $$
Q(\x) = x_1^2 -\alpha_2 x_2^2 -\alpha_3x_3^2.
$$
Then for any fixed $\alpha_2 > 0$ and almost all $\alpha_3 \in [1/2,1],$ Bourgain~\cite{Bourgain16} showed that $\delta(N) \ll N^{-2/5+\ve}$. In the same article, assuming the Lindel\"of hypothesis for the Riemann zeta function, Bourgain improved this to $\delta(N) \ll N^{-1+\ve}$. This result is essentially optimal, as can be seen by simple pigeonhole heuristics (see e.g.~\cite[Page 2]{GGN20}).

Furthermore, Bourgain's result has been generalised to diagonal ternary forms of higher degree by work of Schindler~\cite{Schindler20}. The analogue of Bourgain's theorem for generic ternary indefinite forms was proved by Ghosh and Kelmer~\cite{GK18}. There has been extensive work recently in proving effective results for generic (i.e. a full measure set in the corresponding moduli space) forms, see \cite{GGN20, AM18, KelmerYu20, GH21}. 

In this article, we consider the problem of obtaining quantitative results to ternary Diophantine inequalities on average over a one-parameter family in the spirit of Bourgain and Schindler (op. cit.).

\subsection{Diophantine inequalities over primes}
A central topic in number theory is the investigation of prime solutions to Diophantine equations. Two well-known results of this kind are Vinogradov's theorem~\cite{Vinogradov37}, which states that any odd integer can be written as a sum of three primes, and Hua's~\cite{Hua38} theorem, that any sufficiently large integer $n \equiv 5 \pmod{24}$ can be written as a sum of five squares of primes. We refer the reader to the conjectures of Bourgain, Gamburd and Sarnak~\cite{BGS10} on the existence of prime points on affine varieties. 

A related problem concerns Diophantine approximation in the primes. Given non-zero real numbers $\lambda_1,\lambda_2,\lambda_3$, not all of the same sign, Baker~\cite{Baker67} showed that if at least one of the ratios $\lambda_i/\lambda_j$ is irrational, then there are infinitely many primes $p_i$ such that 
$$
|\lambda_1p_1 + \lambda_2p_2 + \lambda_3p_3| < (\log p)^{-A},
$$ 
where $p = \max_{i = 1,2,3} p_i$, for any $A > 0$ (see also recent work of Matom\"{a}ki~\cite{Matomaki10}). 

For quadratic forms in at least $5$ variables, and higher degree forms in sufficiently many variables, quantitative results were first obtained independently by Ramachandra~\cite{Ramachandra73} and Vaughan~\cite{V74}. Even though these results have since been improved (see~\cite{BH82,CH06}), the number of variables required in the quadratic case has remained at $5$.  

By averaging over a one-parameter family, we show the following quantitative Diophantine approximation theorem for ternary diagonal quadratic forms.
\begin{theorem}\label{thmprime}
Let $A > 0$ be an integer. For any fixed $\alpha_2 > 0$ and almost every $\alpha_3 \in [1/2,1],$ we have
$$
\min_{\substack{p_i \sim N \\ p_i \text{ prime}}}|p_1^2 -\alpha_2 p_2^2 - \alpha_3 p_3^2| \ll_{A,\alpha_2,\alpha_3} (\log N)^{-A}.
$$
\end{theorem}
In addition to the Vinogradov-Korobov zero-free region, the key input in the proof of the theorem is a sharp upper bound for the shifted sum
$$
\sum_{0 \leq l \leq L}\sum_{m \leq M}r(m)r(m+l),
$$
where $r(n)$ is the number of ways of representing an integer as a sum of two squares. Similarly, to obtain a version of Theorem~\ref{thmprime} for degree $k \geq 3$ forms, one would need sharp upper bounds for the sum
$$
\sum_{0 \leq l \leq L}\sum_{m \leq M}r_k(m)r_k(m+l)
$$
where $r_k(m)$ is the number of ways of writing $m$ as a sum of two $k$-th powers, which appears to be a challenging prospect. 

It will be apparent from the proof of Theorem~\ref{thmprime} that it can be stated in slightly greater generality. Indeed, we may replace the primes with elements that lie in any other set $S,$ provided that we understand the distribution of the elements of $S$ in short intervals (i.e. intervals of the form $[N, N+N^{\beta}]$ for some $\beta > 0$), and if we can show that
$$
\sum_{n\in S}w(n/N)n^{it} = o(|[N,2N]\cap S|),
$$
uniformly for $N^{1/10} \leq |t| \leq N^3$ and for $w(x)$ any smooth function with support in $[1/2,2]$. For example, in this way, one can prove a version of Theorem~\ref{thmprime} with the primes replaced with integers that are sums of two squares, but we leave the details to the interested reader.

\subsection{Effective Oppenheim conjecture for inhomogeneous quadratic forms}
Our second theorem considers inhomogeneous ternary forms. An inhomogeneous quadratic form is a quadratic form together with a fixed shift. Namely, let $F$ be a quadratic form and $\bs{\theta} \in \RR^s$ be a vector. Then the inhomogeneous quadratic form is $F(\x + \bs{\theta})$  for $\x \in \RR^s$.
Values of inhomogeneous forms at integer points have been studied extensively  \cite{Marklof03-a, Marklof03-b, MaMo11, StVi20, BG21}. In \cite{GKY20-a, GKY20-b}, Ghosh, Kelmer and Yu proved effective results for inhomogeneous quadratic forms. In \cite{GKY20-a}, they considered the case of a fixed shift, and in the case of a fixed rational shift obtained in addition, an analogue of Bourgain's theorem above. More precisely, the averaging in their result is over all quadratic forms rather than diagonal forms. In \cite{GKY20-b}, they treated the case of a fixed quadratic form and let the shift vary. A fuller account of these recent developments can be found in our upcoming survey. Here we consider ternary diagonal inhomogeneous forms and prove the following quantitative theorem.  

\begin{theorem}\label{thmshift}
Let $k \geq 2$ be an integer and let $F(\x) = x_1^k -\alpha_2x_2^k - \alpha_3 x_3^k$ and $\ve > 0$. Let $\bs{\theta} = (\theta_1,\theta_2,\theta_3) \in \RR^3$ be a fixed vector. Then for any fixed $\alpha_2 >0$ and almost every $\alpha_3 \in [1/2,1]$ the following statements hold:
\begin{enumerate}
\item Assume the exponent pair conjecture. Then we have
$$
\min_{\substack{\x \in \ZZ^3 \\ |\x| \sim N}}|F(\x+\bs{\theta)}| \ll_{\ve,\alpha_2,\alpha_3,\bs{\theta}} N^{k-3+\ve}.
$$
\item Unconditionally, we have
$$
\min_{\substack{\x \in \ZZ^3 \\ |\x| \sim N}}|F(\x+\bs{\theta)}| \ll_{\ve,\alpha_2,\alpha_3,\bs{\theta}} N^{k-12/5+\ve}.
$$
\end{enumerate}
\end{theorem}
We will now briefly recall the exponent pair conjecture. For a more detailed account, see e.g.~\cite[\S 8.4]{IK04}. Let $e(x) = \exp (2\pi i x)$. For real numbers $p,q$, such that $0 \leq p,q \leq 1/2$, we say that $(p,q)$ is an exponent pair if for any fixed $\ve > 0$ we have the bound
\begin{equation*}
\sum_{N \leq n \leq N'} e(f(n)) \ll_{\ve} F^pN^{q-p+1/2}F^{\ve},
\end{equation*}
 for any $2 \leq N \leq N' \leq 2N$, with $f(x)$ a smooth function on $[N,2N]$ such that there exists a positive real number $F \geq N$ with the property that for any $j \geq 0$ we have
\begin{equation*}
FN^{-j} \ll_j |f^{(j)}(x)| \ll_j FN^{-j}
\end{equation*}
for any $x \in [N,2N]$. The Exponent Pair Conjecture is the assertion that $(0,0)$ is an exponent pair. This conjecture, for example, implies the truth of the Lindel\"of hypothesis for the Riemann zeta function $\zeta(s)$.

In the course of proving Theorem~\ref{thmshift}, we also prove sharp upper bounds for the following Diophantine inequality in four variables, which is reminiscent of the problem of simultaneous inhomogeneous Diophantine approximation on curves. 
\begin{theorem}\label{prop1}
Let $\theta_1,\theta_2,\alpha, \beta$ be fixed real numbers. Let $0 < \delta < 1$. Define
$\mathcal{N}(M,\alpha,\delta)$ to be the number of solutions $(m_1,m_2,m_3,m_4) \in \ZZ^4 \cap [M,2M]^4$ to the inequality 
\begin{equation}\label{eq:nmaddef}
|(m_1+\theta_1)^\alpha -(m_2+\theta_1)^\alpha + \beta(m_3+\theta_2)^\alpha -  \beta(m_4+\theta_2)^\alpha| \leq \delta M^{\alpha}.
\end{equation}
Suppose that $\alpha \neq 0,1$. Then for any $\ve > 0$ we have $$\mathcal{N}(M,\alpha,\delta) \ll_{\alpha,\beta,\ve,\theta_1,\theta_2} M^{2+\ve}+\delta M^{4+\ve}.$$
\end{theorem}
The result with $\theta_1 = \theta_2 = 0$ is a well-known result of Robert and Sargos~\cite[Theorem 2]{RS06} and the proof of Theorem~\ref{prop1} is an adaptation of their argument. We should remark here that while the bound in the theorem with $\theta_1 = \theta_2 = 0$ can be deduced from recent work of Huang~\cite{Huang20}, it appears much more difficult to do this in the case where $\theta_i \neq 0$.

\subsection{Averaging over two-parameter families}
In this article, we have considered one-parameter families of ternary diagonal forms. One can ask if it might be possible to prove stronger results in Theorems~\ref{thmprime} and~\ref{thmshift} by also averaging over $\alpha_2$ (as in~\cite{Bourgain16} and~\cite{Schindler20}). In order to exploit this extra averaging, we will require sharp upper bounds for 
$$
\int_{T}^{2T}|\sum_{n \sim N}\Lambda(n)n^{it}|^2\, dt,
$$
in the case of Theorem~\ref{thmprime}, and for
$$
\int_{T}^{2T}|\sum_{n \sim N}(n+\theta_3)^{it}|^2\, dt,
$$ 
in the case of Theorem~\ref{thmshift}, with $1 \ll T \ll N^{3}$. One could analyse the integrals using the Montgomery-Vaughan inequality~\cite[Corollary 3]{MV74}, but this is effective only when $T \gg N$, and we are unable to estimate them satisfactorily in the complementary range. We conclude our introduction by remarking that if $\theta_3 = a/q$ is a rational number, then by splitting $n$ into arithmetic progressions modulo $q$, one can nevertheless prove the bound $\int_{T}^{2T}|\sum_{n \sim N}(n+\theta_3)^{it}|^2\, dt \ll_{\theta_3,\ve} TN^{1+\ve}$, which will enable us to prove an optimal version of Theorem~\ref{thmshift} by averaging over both the parameters $\alpha_2$ and $\alpha_3.$

\begin{ack}
We are grateful to the referee for a careful reading of our manuscript, and for several comments that have improved the exposition of the article.
\end{ack}

\section{A Diophantine inequality over the primes}

We now turn to the proof of Theorem~\ref{thmprime}. Using the Borel-Cantelli lemma, it is sufficient to prove the following 
\begin{proposition}\label{propprime}
Fix $\alpha_2 > 0$. Let $N$ be a large parameter. Assume that $\delta < 1$. The inequality 
$$|x_1^2 -\alpha_2 x_2^2 -\alpha_3x_3^2| < \delta$$
has a non-trivial solution with $x_i \sim N$ and $x_i$ prime, for all $\alpha_3 \in [1/2,1]$ except for a set of Lebesgue measure at most 
$$\ll \delta^{-1}(\log N)^8e^{-c(\log N)^{1/3}(\log \log N)^{-1/3}},$$ for some $c > 0$. 
\end{proposition}
Let $\omega_0(x)$ and $\omega(\x) = \prod_{i=1}^3\omega_i(x_i)$ be smooth functions as in~\cite[Section 2]{Schindler20}. In particular, $\omega_0(x)$ is an odd function.

We seek a lower bound for the sum
$$
\sum_{\substack{\x \in \ZZ^3 \\ x_i \text{ prime}}}\omega(\x/N) \bs{1}_{|x_1^2 - \alpha_2x_2^2 - \alpha_3x_3^2| < \delta}.
$$
Since $|\log (m/n)| \geq |m-n|/(m+n)$, it suffices to obtain a lower bound for the sum
$$
S_1(\alpha_3) = \sum_{\substack{\x \in \ZZ^3 \\ x_i \text{ prime}}}\omega(\x/N) \bs{1}_{|\log(x_1^2 - \alpha_2x_2^2) - \log(\alpha_3x_3^2)| < c_0\delta N^{-2}}
$$
for $c_0 > 0$. Define the exponential sums
$$
F_1(t) = \sum_{\substack{\x \in \ZZ^2 \\ x_i \text{ prime}}}\omega_1(x_1/N)\omega_2(x_2/N)e^{it \log(x_1^2 - \alpha_2x_2^2)}
$$
and
$$
F_2(t) = \sum_{\substack{p \in \ZZ \\ p  \text{ prime}}}\omega_3(p/N)e^{it \log p}.
$$
Let $T = \frac{2N^2}{c_0\delta}$ and define
$$
S_2(\alpha_3) = \frac{1}{T}\int_{-\infty}^{\infty}\widehat{\omega_0}(t/T)F_1(t)F_2(-2t)e^{-it \log \alpha_3}\, dt.
$$
Then we have $S_1(\alpha_3) \geq S_2(\alpha_3)$. Let $0 < \theta < 1$ be a real number that we will specify shortly. Split 
$$
\widehat{\omega_0}(t/T) = \widehat{\omega_0}(t/N^{1-\theta}) + \left(\widehat{\omega_0}(t/T) - \widehat{\omega_0}(t/N^{1-\theta}\right).$$
 Let $S_3(\alpha_3)$ denote the contribution from the former and $S_4(\alpha_3)$ the contribution from the latter to $S_2(\alpha_3)$. As in~\cite{Bourgain16} and~\cite{Schindler20}, our strategy will be to obtain a lower bound for $S_3(\alpha_3)$ and an upper bound for the set of $\alpha_3$ such that $S_4(\alpha_3)$ is `large'.

\subsection{A lower bound for $S_3(\alpha)$}
To obtain a lower bound for $S_3(\alpha)$, we will require information on the distribution of primes in short intervals. 

Let $\pi(x)$ denote the number of primes not exceeding $x$. By work of Huxley~\cite{Huxley72}, we know that if $\theta > 7/12$, then there exists $x(\theta)$ such that for all $x \geq x(\theta)$ we have
$$
\pi(x+x^{\theta}) - \pi(x) \sim \frac{x^{\theta}}{\log x}.
$$
For the rest of this section, we will fix $\theta = 2/3$. Then we have the following result.
\begin{lemma}
Let $\theta = 2/3.$ Then there exists a constant $c_2 > 0$ such that for all $\alpha_3$ in $[1/2,1]$ we have 
$$S_3(\alpha_3) \geq  c_2 \delta N/(\log N)^3.$$
\end{lemma}
\begin{proof}
We have
$$
S_3(\alpha_3) \geq \frac{N^{1-\theta}}{T}\sum_{\substack{\x \in \ZZ^3 \\ x_i \text{ prime}}}\omega(\x/N)\bs{1}_{[|\log (x_1^2 - \alpha_2 x_2^2) - \log (\alpha_3x_3^2)| < N^{-1+\theta}]}.
$$
Hence there exists a constant $c_1 > 0$ such that 
$$
S_3(\alpha_3) \geq \frac{N^{1-\theta}}{T}\sum_{\substack{\x \in \ZZ^3 \\ x_i \text{ prime}}}\omega(\x/N)\bs{1}_{[|x_1^2 - \alpha_2 x_2^2 - \ \alpha_3x_3^2| < c_1N^{1+\theta}]}.
$$
We have
\begin{align*}
\sum_{\substack{\x \in \ZZ^3 \\ x_i \text{ prime}}}\omega(\x/N)&\bs{1}_{[|x_1^2 - \alpha_2 x_2^2 - \ \alpha_3x_3^2| < c_1N^{1+\theta}]} = \sum_{\substack{\x \in \ZZ^3 \\ x_i \text{ prime}}}\omega_2(x_2/N)\omega_3(x_3/N) \\
&\quad\times\sum_{\substack{\alpha_2 x_2^2 + \alpha_3 x_3^2 - c_1N^{1+\theta} \leq x_1^2 \leq \alpha_2x_2^2 +\alpha_3x_3^2 + c_1N^{1+\theta} \\ x_1 \text{ prime}}}\omega_1(x_1). 
\end{align*}
Since $\theta = 2/3 > 7/12$, we see that the inner sum is $\gg N^{\theta}/\log N$ for any fixed $x_2$ and $x_3$. Summing over $x_2$ and $x_3$ completes the proof of the lemma.
\end{proof}

\subsection{Reduction to a counting problem} 
Using Chebyshev's inequality, we get
\begin{align*}
\meas \left\{\alpha_3 \in [1/2,1] : |S_4(\alpha_3)| \geq \frac{c_2}{2} \frac{\delta N}{(\log N)^3}\right\} &\leq 4c_2^{-2}\delta^{-2}(\log N)^6N^{-2}\times \\
&\quad \quad \int_{1/2}^1|S_4(\alpha_3)|^2\, d\alpha_3.
\end{align*}
Since $\omega_0$ was chosen to be an odd function, there exists a constant $c_3$ such that
$$
\bigg\vert\omega_0\left(\frac{t}{T}\right) - \omega_0\left(\frac{t}{N^{1/3}}\right)\bigg\vert \leq c_3 \min\left(1, \frac{t^2}{N^{2/3}},\left(\frac{T}{|t|}\right)^{10}\right).
$$
By Parseval's theorem, we get
$$
\int_{1/2}^1 |S_4(\alpha_3)|^2\, d\alpha_3 \leq c_3^{2}T^{-2}\int_{-\infty}^{\infty}\min\left(1, \frac{t^4}{N^{4/3}},\left(\frac{T}{|t|}\right)^{20}\right)|F_1(t)|^2|F_2(t)|^2\, dt.
$$
Estimating $F_1(t)$ and $F_2(t)$ trivially over the interval $|t| \leq N^{1/10}$, and using~\cite[Lemma 2.2]{Schindler20} in the complementary range we find that
$$
\meas \left\{\alpha_3 \in [1/2,1] : |S_4(\alpha_3)| \geq (c_2/2)\delta N/(\log N)^3\right\} \leq c_4(\log N)^6 N^{-5/6} + I_2,
$$
where
\begin{equation}\label{eq:pf2i4}
I_2 \leq N^{-6}(\log N)^7 \max_{|t| \geq N^{1/10}}\left(\min\left(1,\left(\frac{T}{|t|}\right)\right)|F_2(t)|\right)^2\sup_{1/\log N \ll U \ll T}I_4(U)
\end{equation}
with 
$$
I_4(U) = U\sum_{\substack{\y \in \ZZ^4 \\ y_i \text{ prime}}}\widetilde{\omega}(\y/N)\bs{1}_{[|y_1^2-y_2^2-\alpha_2(y_3^2-y_4)^2|\ll N^2/U}.
$$
Applying Theorem~\ref{prop1} with $\alpha = 2$ and $\theta_i = 0$, we easily obtain the bound 
$$
I_4(U) \ll (U+N^2)N^{2+\ve}
$$
for any $\ve > 0$. In the next result, we show that we may remove the $N^{\ve}$ factor at the expense of a power of $\log N$. 
\begin{lemma}\label{lemmadiop}
For $1/\log N \leq U \leq T$, we have
$$
I_4(U) \ll \left(UN^2 + N^4\right)(\log N).
$$
\end{lemma}
\begin{proof}
To prove the lemma, we may drop the constraint that $y_i$ are prime. Suppose first that $U \ll N$. Fixing $y_2,y_3$ and $y_4$ we find that $x_1$ is constrained by the inequality 
$$
|y_1^2-y_2^2-\alpha_3(y_3^2-y_4^2)| \ll N^2/U,$$ 
which forces $y_1^2$ to lie in an interval of length $\ll N^2/U$. As a result, we get that 
$$
I_4(U) \ll UN^3(N/U) \ll N^4,
$$
which is sufficient. So we may suppose for the rest of the argument that $U \gg N$. Since $\widetilde{\omega}$ is non-negative, it suffices to show that 
$$
\#\left\{y_i \sim N : |y_1^2-y_2^2-\alpha_2(y_3^2-y_4^2)|\ll N^2/U\right\} \ll N^2(1+N^2/U)(\log N).
$$
We will now proceed as in the proof of~\cite[Lemma 7]{TB03}. We are grateful to Tim Browning for suggesting its use in the proof of this lemma. For $\beta < N$ put 
$$
N(\beta,N) = \#\left\{\y \in \ZZ^4 \cap [N,2N]^4 : |y_1^2-y_2^2-\alpha_2(y_3^2-y_4^2)|\leq \beta \right\}
$$
and 
$$
N(n;1,2) = \#\left\{\y \in \ZZ^2 \cap [N,2N]^2 : n\beta \leq y_1^2-y_2^2 \leq (n+3)\beta\right\}.
$$
Defining $N(n;3,4)$ similarly, we have that
\begin{align*}
N(\beta,N) &\leq \sum_{n=1}^{\infty}N(n;1,2)\left\{N(-(n+4);3,4)+N(-(n+1);3,4)\right\} \\
&\ll \left(\sum_{n=1}^{\infty}N(n;1,2)^2\right)^{1/2}\left(\sum_{n'=1}^{\infty}N(-n';3,4)^2\right)^{1/2},
\end{align*}
by the Cauchy-Schwarz inequality. Thus setting 
$$
M(\beta,N) = \#\left\{N \leq y_1,y_2,z_1,z_2 \leq 2N : |(y_1^2-y_2^2-(z_1^2-z_2^2)|\right\} \leq 3\beta
$$
we find that 
$$
N(\beta,N) \ll M(\beta,N)^{1/2}M(\beta/\alpha_2,N)^{1/2}.
$$
Let $F(\x) = x_1^2-x_2^2-x_3^2+x_4^2$. Then we have
\begin{align*}
M(\beta,N) &= \sum_{\substack{m \in \ZZ \\ |l| \leq 3\beta}}\sum_{\substack{x_i \sim N \\ F(\x) = l}}1 \leq \sum_{\substack{0 \leq l \leq 3\beta}}\sum_{m \leq 8N^2}r(m)r(m+l),
\end{align*}
where $r(m)$ is the number of ways of writing an integer as the sum of two squares. The lemma now follows by adapting the proof of~\cite[Theorem 1]{BB06}. 

For a more direct proof, we may appeal to~\cite[Theorem 1.4]{K18} to get
\begin{align*}
M(\beta,N) &\ll N^2\log N + \sum_{1 \leq l \leq 3\beta}d(l)N^2 + O\left((1+\beta)N^{19/20+\ve}(1+\beta/N^2)^{3/20+\ve}\right)\\
&\ll N^2\log N + (1+\beta)\log (2+\beta)N^2,
\end{align*}
since $\beta < N$, by assumption. The lemma follows upon setting $\beta = N^2/U$.
\end{proof}

\subsection{Proof of Proposition~\ref{propprime}}
Using a zero-free region for the Riemann zeta function, we will now prove an upper bound for $F_2(t)$. Proposition~\ref{propprime} will then follow from~\eqref{eq:pf2i4}, Lemma~\ref{lemmadiop} and the following lemma since
$$
F_2(t) = \sum_{n=1}^{\infty}\omega_3(n/N)\Lambda(n)n^{it} + O(N^{1/2}\log N),
$$
where $\Lambda(n)$ denotes the von Mangoldt function. 
\begin{lemma}
Let $w(x)$ be a smooth function with support in $[1/2,2]$. Suppose that $N^{1/10} \leq |t| \leq N^3.$ Then there exists a real number $c > 0$ such that
\begin{equation*}
\sum_{n=1}^{\infty} w(n/N)\Lambda(n)n^{it} \ll N\exp\left(-c(\log N)(\log |t|)^{-2/3}(\log \log |t|)^{-\frac{1}{3}}\right). 
\end{equation*}
\end{lemma}
\begin{proof}
By Mellin inversion we have 
$$
\sum_{n=1}^{\infty} w(n/N)\Lambda(n)n^{it} = -\frac{1}{2\pi i}\int_{(\sigma_0)}\widetilde{w}(s)N^s\frac{\zeta'}{\zeta}(s-it)\, ds
$$
for $\sigma_0 = 1 + 1/\log N.$ Let $Y = |t|^{1/10}.$ Since $\widetilde{w}(s) \ll_A (1+|s|)^{-A}$ for $0 < \Re (s) < 2$, we have
$$
\int_{\sigma_0 \pm iY}^{\sigma_0 \pm i \infty} \widetilde{w}(s)N^s\frac{\zeta'}{\zeta}(s-it)\, ds \ll_A \frac{N(\log N)^2}{Y^A}.
$$
Thus we get for any $A > 0$ that
$$
\sum_{n=1}^{\infty} w(n/N)\Lambda(n)n^{it} = -\frac{1}{2\pi i}\int_{\sigma_0-iY}^{\sigma_0+iY}\widetilde{w}(s)N^s\frac{\zeta'}{\zeta}(s-it)\, ds + O_A\left(\frac{N(\log N)^2}{Y^A}\right).
$$
We will now move the line of integration to 
$$
\sigma_1 = 1 - c(\log 2|t|)^{-\frac{2}{3}}(\log \log 2|t|))^{-\frac{1}{3}}
$$ for $c > 0$ so that we may apply~\cite[Theorem 8.29]{IK04} to ensure that $\zeta(s-it) \neq 0$ for $\Re(s) \geq \sigma_1$ and $\Im(s) \leq Y$. In this process, we pick up a pole at $s = 1+it$ to get
\begin{align*}
\sum_{n=1}^{\infty} w(n/N)\Lambda(n)n^{it}  &= \widetilde{w}(1+it)N^{1+it}  - \frac{1}{2\pi i} \int_{\sigma_0 - iY}^{\sigma_1 - iY} \widetilde{w}(s)N^s\frac{\zeta'}{\zeta}(s-it)\, ds\\
&\quad \quad + \frac{1}{2\pi i}\int_{\sigma_0 + iY}^{\sigma_1 + iY} \widetilde{w}(s)N^s\frac{\zeta'}{\zeta}(s-it)\, ds \\
&\quad \quad  + \frac{1}{2\pi i}\int_{\sigma_1 - iY}^{\sigma_1 +iY} \widetilde{w}(s)N^s\frac{\zeta'}{\zeta}(s-it)\, ds.
\end{align*}
Since $|t| \geq N^{1/10}$ by hypothesis, we see that the first term is bounded by $O_A(N|t|^{-A}) = O(N^{1-A/10}),$ for any $A > 0$. Using~\cite[Theorem 8.29]{IK04} to bound $\frac{\zeta'}{\zeta}(s-it)$ we get that 
\begin{align*}
\int_{\sigma_0 \pm iY}^{\sigma_1 \pm iY} \widetilde{w}(s)N^s\frac{\zeta'}{\zeta}(s-it)\, ds &\ll (\sigma_0 - \sigma_1)\frac{N}{Y^A}(\log |t|)^{\frac{2}{3}}(\log \log |t|)^{\frac{1}{3}} \\
&\ll \frac{N(\log N)^2}{Y^A},
\end{align*}
since $|t| \leq N^3$, by assumption. Finally, using the bound $\widetilde{w}(s) \ll |s|^{-1}$ we get
\begin{align*}
\int_{\sigma_1 - iY}^{\sigma_1 +iY} \widetilde{w}(s)N^s\frac{\zeta'}{\zeta}(s-it)\, ds &\ll N^{1-c(\log 2|t|)^{-\frac{2}{3}}(\log \log 2|t|)^{-\frac{1}{3}}} \times \\
&(\log |t|)^{\frac{2}{3}}(\log \log |t|)^{\frac{1}{3}}\int_{-Y}^Y\frac{dy}{\sqrt{1+|y|^2}}. 
\end{align*}
Putting together all the estimates completes the proof of the lemma.
\end{proof}

\section{Quantitative Oppenheim conjecture for inhomogeneous diagonal forms}

This section is devoted to the proof of Theorem~\ref{thmshift}, which will follow from the following results via an application of the Borel-Cantelli lemma.
\begin{proposition}\label{propshift}
Fix $\alpha_2 > 0$, let $k \geq 2$ be an integer. Let $\bs{\theta} \in \RR^3.$ Let $N$ be a large real parameter. Assume that $\delta > N^{k-3}$. \begin{enumerate}
\item Suppose that the Exponent Pair conjecture is true, then the inequality 
$$|(x_1+\theta_1)^k -\alpha_2 (x_2+\theta_2)^k -\alpha_3(x_3+\theta_3)^k| < \delta$$
has a non-trivial solution with $|\x| \sim N$ for all $\alpha_3 \in [1/2,1]$ except for a set of Lebesgue measure at most 
$$
\ll N^{-1+\ve} + N^{k-3+\ve}\delta^{-1}.$$
\item The statement above holds unconditionally except for a set of measure at most
$$
\ll N^{-1+k/3+\ve}\delta^{-13} + N^{4k/3-3+\ve}\delta^{-4/3}.
$$
\end{enumerate}
\end{proposition}
\begin{proposition}\label{12/5bound}
Fix $\alpha_2 > 0$. Let $k \geq 2$ and let $\bs{\theta} \in \RR^3$ be a fixed vector. Suppose that $N^{k-3} < \delta < N^{k-2}$. Then the inequality
$$
|(x_1+\theta_1)^k - \alpha_2(x_2+\theta_2)^k -\alpha_3 (x_3+\theta_3)^k| < \delta
$$
has a non-trivial solution with $|\x| \sim N$ for all $\alpha_3 \in [1/2,1]$ with the exception of a set of $\alpha_3$ of measure at most
$$
\ll N^{5k/6-2+\ve}\delta^{-5/6} + N^{10k/9-8/3+\ve}\delta^{-10/9}.
$$
\end{proposition}
\subsection{Reduction to a counting problem}
Let $\omega_0(x)$ and $\omega(\x) = \prod_{i=1}^3\omega_i(x_i)$ be smooth functions as in~\cite[Section 2]{Schindler20}. Proceeding as in~\cite{Schindler20}, we see that to prove Theorem~\ref{thmshift}, it will suffice to obtain a lower bound for the sum
\begin{equation*}
\sum_{\x \in \ZZ^3}\omega(\x/N)\bs{1}_{[|F(\x+\bs{\theta})|<\delta]},
\end{equation*}
which will follow from a lower bound for the sum
\begin{equation*}
S_1(\alpha_3) = \sum_{\x \in \ZZ^3}\omega(\x/N)\bs{1}_{[|\log((x_1+\theta_1)^k - \alpha_2(x_2+\theta_2)^k)-\log (\alpha_3(x_3+\theta_3)^k)| < c_0\delta/N^{k}]},
\end{equation*}
where $c_0$ is a constant that depends only on $\bs{\theta}$, $\omega(\x)$ and $\alpha_2$. Let 
$$F_1(t) = \sum_{\x \in \ZZ^2}\omega_1(x_1/N)\omega_2(x_2/N)e^{it\log((x_1+\theta_1)^k - \alpha_2(x_2+\theta_2)^k)}$$ 
and 
$$F_2(t) = \sum_{x_3 \in \ZZ}\omega_3(x_3/N)e^{it\log (x_3+\theta_3)}.$$
Set $T = \frac{2N^k}{c_0\delta}$ and define
\begin{equation*}
S_2(\alpha_3) = \frac{1}{T}\int_{-\infty}^{\infty}\widehat{\omega_0}(t/T)F_1(t)F_2(-kt)e^{-it\log \alpha_3}\, dt.
\end{equation*} 
Then we have $S_1(\alpha_3) \geq S_2(\alpha_3)$. Write $S_2(\alpha_3) = S_3(\alpha_3) + S_4(\alpha_3),$ where $S_3$ and $S_4$ are as in~\cite[Equations (2.4), (2.5)]{Schindler20}. A straightforward adaptation of~\cite[Lemma 2.1]{Schindler20} shows that for any $\alpha_3 \in [1/2,1]$, we have the lower bound
\begin{equation}\label{eq:s3alpha3}
S_3(\alpha_3) \gg N^{3-k}\delta,
\end{equation}
where the implied constant depends only on $\alpha_2$ and the function $\omega(\x)$. Our task for the rest of the section will be to get an upper bound for $S_4(\alpha_3)$.

If $c_2$ is the implicit constant in~\eqref{eq:s3alpha3}, by Chebyshev's inequality we find that
$$\meas \left\{\alpha_3 \in [1/2,1] : |S_4(\alpha_3)| \geq \frac{c_2}{2}N^{3-k}\delta\right\} 
\leq 
4c_2^{-2}\delta^{-2}N^{2k-6}\int_{1/2}^1|S_4(\alpha_3)|^2\, d\alpha_3.$$
Let 
$$\widetilde{\omega}(y_1,y_2,y_3,y_4) = \omega_1(y_1)\omega_2(y_2)\omega_1(y_3)\omega_2(y_4)$$ 
and define
$$I_4(U) = U\sum_{\y \in \ZZ^4}\widetilde{\omega}(\y/N)1_{[|(y_1+\theta_1)^k-\alpha_2(y_2+\theta_2)^k-((y_3+\theta_1)^k-\alpha_2(y_4+\theta_2)^k)|\ll N^k/U]}.$$
Then by~\cite[Equations (2.7)--(2.9) \& Lemma 2.2]{Schindler20} we see that there exists a constant $c_4$ such that
\begin{equation}\label{eq:exmeasurelindelof}
\meas \left\{\alpha_3 \in [1/2,1] : |S_4(\alpha_3)| \geq (c_2/2)N^{3-k}\delta\right\} \leq c_4N^{-3/2} + I_2,
\end{equation}
where
\begin{equation}\label{eq:f2i4}
I_2 \leq N^{-6}\log N \max_{|t| \geq N^{1/10}}\left(\min\left(1,\left(\frac{T}{|t|}\right)\right)|F_2(t)|\right)^2\sup_{1/\log N \ll U \ll T}I_4(U).
\end{equation}

\begin{lemma}\label{lemmaDiophantine}
\begin{equation*}
I_4(U) \ll N^{4+\ve} + N^{k+2+\ve}/\delta.
\end{equation*}
\end{lemma}
\begin{proof}
If $U \leq 1$, then the result is trivial. So we may suppose that $U > 1.$ In this case, a direct application of Theorem~\ref{prop1} yields 
$$I_4(U) \ll U\left(N^{2+\ve} + N^{4+\ve}/U\right).$$
Since $U \ll T = 2N^{k}/c_0\delta$, we see that the lemma holds.
\end{proof}

\subsection{Proof of the main propositions}
We can now complete the proof of Propositions~\ref{propshift} and~\ref{12/5bound}. We will need the following estimate for $F_2(t)$.
\begin{lemma}\label{lemmalindelof}
Let $w(x)$ be a smooth function with support in $[1/4,2]$. For $\theta$ a real number define 
$$S(N) = \sum_{n \in \ZZ}w(n/N)(n+\theta)^{it}.$$
Suppose that $|t| \geq N^{\frac{1}{10}}$. Suppose the exponent pair conjecture holds. Then for any $\ve > 0$ we have
$$S(N) \ll_{\theta,\ve} N^{\frac{1}{2}+\ve}(1+|t|)^{\ve}.$$
Unconditionally, we have
$$
S(N) \ll_{\theta,\ve} N^{\frac{1}{2}+\ve}(1+|t|)^{\frac{1}{6}+\ve}.
$$
\end{lemma}
\begin{proof}
Write $\theta = [\theta] + \left\{\theta\right\}$, where $[\theta]$ is the integer part and $\left\{\theta\right\}$ is the fractional part of $\theta$.  Then we have
\begin{align*}
S(N) &= \sum_{m \in \ZZ}w\left(\frac{m-[\theta]}{N}\right)(m+\left\{\theta\right\})^{it} \\
&= \sum_{m \in \ZZ}w(m/N)(m+\left\{\theta\right\})^{it} + O_{\theta}(1).
\end{align*}
If $\left\{\theta\right\} = 0,$ then the lemma follows from~\cite[Equation (2.6)]{Bourgain16} and the remark following the equation. So for the rest of the proof we may assume that $0 < \theta < 1.$ 

For $\Re(s) > 1$ let
$
\zeta(s,\theta) = \sum_{n=0}^{\infty}(n+\theta)^{-s}
$
denote the Hurwitz zeta function. Note that if $\theta = 1$, we have $\zeta(s,1) = \zeta(s)$, the Riemann-zeta function. The function $\zeta(s,\theta)$ has a simple pole with residue $1$ at $s = 1$, and has analytic continuation to the entire complex plane. Using Mellin inversion, we may write
\begin{align*}
S(N) = \frac{1}{2\pi i}\int_{3/2-i\infty}^{3/2+i\infty}\widetilde{w}(s)\zeta(s-it,\theta)N^s \, ds.
\end{align*}

Note that by repeatedly integrating by parts, we have the bound $\widetilde{w}(s) \ll_{A} (1+|s|)^{-A}$ for any $A > 0$ in the region $0 < \Re(s) < 2$. Moving the line of integration to $\Re(s) = 1/2$, we pick up a pole at $s=1$ to obtain
\begin{align}\label{eq:sn}
S(N) &= \frac{1}{2\pi i}\int_{-\infty}^{\infty}N^{1/2+ix}\zeta(1/2+i(x-t),\theta)\widetilde{w}(1/2+ix)\, dx + N^{1+it}\widetilde{w}(1+it).
\end{align}
Using partial summation, for $1/2 \leq \Re(s) \leq 3/2$ we have
\begin{equation*}
\zeta(\sigma + it,\theta) = \sum_{n \leq |t|}(n+\theta)^{-\sigma - it} + O((1+|t|)^{-\sigma}). 
\end{equation*}
Furthermore, by a dyadic argument, we obtain
\begin{align*}
\zeta(1/2 + it, \theta) \ll (1+|t|)^{\ve}\max_{1 \leq K \leq |t|} \frac{1}{\sqrt{K}} \bigg\vert\sum_{K \leq n \leq 2K} (n+\theta)^{-it} \bigg\vert+ O((1+|t|)^{-1/2}).
\end{align*}
If $(p,q)$ is an exponent pair (see e.g.~\cite[\S 8.4]{IK04}), we get that
$$
\sum_{K \leq n \leq 2K} (n+\theta)^{-it} \ll  |t|^{p}K^{q-p+1/2}(|t|K)^{\ve}
$$
for any $K \leq |t|$. From this we get the bound $\zeta(1/2 + it, \theta) \ll_{\ve,\theta} (1+|t|)^{\ve}$ by assuming the Exponent Pair Conjecture, which is the assertion that $(0,0)$ is an exponent pair. And by using the exponent pair $(1/6,1/6)$ we get the bound $\zeta(1/2+it, \theta) \ll (1+|t|)^{1/6+\ve}$. Inserting these bounds into~\eqref{eq:sn}, and using the decay properties of $\widetilde{w}$ (recall that $|t| \geq N^{1/10}$) completes the proof of the lemma.
\end{proof}
\begin{proof}[Proof of Proposition~\ref{propshift}]
Inputting the bounds from Lemmas~\ref{lemmaDiophantine} and~\ref{lemmalindelof} into \eqref{eq:f2i4}, we see from~\eqref{eq:exmeasurelindelof} that 
$$
\meas \left\{\alpha_3 \in [1/2,1] : |S_4(\alpha_3)| \geq (c_2/2)N^{3-k}\delta\right\} \leq c_4 N^{-3/2} + 
N^{-1+\ve} + N^{k-3+\ve}/\delta,
$$
which completes the proof of Proposition~\ref{propshift}.
\end{proof}

\begin{proof}[Proof of Proposition~\ref{12/5bound}]
The proposition will follow by adapting the proof of~\cite[Proposition 4.4]{Schindler20}. Therefore, it will suffice to establish versions of Lemmas 4.2 and 4.3 in loc. cit. It is easy to see that the proof of Lemma 4.3 in loc. cit. carries through in our case with only minor modifications, and in the proof of Lemma 4.3, the only difference in our case is that we will use Lemma~\ref{lemmaDiophantine} in place of (3.2) in Schindler's work. With these changes in place, the proposition follows by the formalism laid out in~\cite[\S 3]{Bourgain16}.
\end{proof}

\section{A Diophantine inequality in four variables}

Our aim in this section will be to prove Theorem~\ref{prop1}. Given two functions $\phi_1(x)$ and $\phi_2(x),$ let $\mathcal{N}_{\phi_1,\phi_2}(M,\alpha, \delta)$ denote the number of solutions $(m_1,m_2,m_3,m_4)\in \ZZ^4 \cap [M,2M]^4$ such that $$|\phi_1(m_1)-\phi_1(m_2)+\phi_2(m_3)-\phi_2(m_4)| \leq \delta M^{\alpha}.$$ 
If \begin{align*}
\phi_1(x) = (x+\theta_1)^\alpha \, \, \, \, \text{and}\, \, \,\, \phi_2(x) = \beta(x+\theta_2)^{\alpha},
\end{align*}
then for this choice of functions, we have $\mathcal{N}_{\phi_1,\phi_2}(M,\alpha,\delta) = \mathcal{N}(M,\alpha,\delta)$, where $\mathcal{N}(M,\alpha,\delta)$ was defined in~\eqref{eq:nmaddef}. By arguing as in the proof of Lemma~\ref{lemmadiop}, it suffices to prove Theorem~\ref{prop1} with $\beta = 1$, which we will assume from now on.

\subsection{Some lemmas}
In this section, we will prove some lemmas that go into the proof of Theorem~\ref{prop1}. These lemmas are based on~\cite[Lemmas 1--7]{RS06}, and are proved along similar lines. We begin with a result which expresses $\mathcal{N}_{\phi_1,\phi_2}(M,\alpha,\delta)$ in terms of exponential sums. For convenience, where there is no ambiguity, we write $\mathcal{N}_{\phi_1,\phi_2}(M,\delta)$ in place of $\mathcal{N}_{\phi_1,\phi_2}(M,\alpha,\delta).$

\begin{lemma}\label{ulbound}
Let $X > 0$ and $c\geq 0$ be real numbers. Suppose that $a_n$ and $b_n$ are complex numbers such that $|a_n|, |b_n| \leq 1$. Le $\phi_1(x)$ and $\phi_2(x)$ be real valued functions. Then 
\begin{equation}\label{eq:lowerbound}
\begin{split}
\frac{1}{X}\int_c^{c+X}&\bigg\vert\sum_{m = M+1}^{2M} a_n e\left(\frac{x\phi_1(n)}{M^{\alpha}}\right)\bigg\vert^2 \times \\
&\quad \quad \quad \quad  \bigg\vert\sum_{n= M+1}^{2M} b_ne\left(\frac{x\phi_2(n)}{M^{\alpha}}\right)\bigg\vert^2 \, dx \leq \frac{\pi^2}{4}\mathcal{N}_{\phi_1,\phi_2}(M,2/X)
\end{split}
\end{equation}
and
\begin{equation}\label{eq:upperbound}
\mathcal{N}_{\phi_1,\phi_2}(M,1/X) \leq \frac{\pi^2}{X}\int_0^{X/2}\bigg\vert\sum_{n = M+1}^{2M} e\left(\frac{x\phi_1(n)}{M^{\alpha}}\right)\bigg\vert^2 \bigg\vert\sum_{n = M+1}^{2M} e\left(\frac{x\phi_2(n)}{M^{\alpha}}\right)\bigg\vert^2 \, dx. 
\end{equation}
\end{lemma}
\begin{proof}
Let $\Lambda(a) = \max (0,1-|a|)$. Then the Fourier transform of $\Lambda(a)$ is given by $$\sinc^2 (a) =\begin{cases} (\sin \pi a)^2/(\pi a)^2 &\mbox{ if $a \neq 0,$} \\ 1 &\mbox{ if $a = 0.$}\end{cases}$$For $\delta > 0$ consider the integral
\begin{equation*}
I = \delta\int_{-\infty}^{\infty} \sinc^2(\delta (x-c)) \bigg\vert\sum_{n = M+1}^{2M} a_n e\left(\frac{x\phi_1(n)}{M^{\alpha}}\right)\bigg\vert^2 \bigg\vert\sum_{n = M+1}^{2M} b_ne\left(\frac{x\phi_2(n)}{M^{\alpha}}\right)\bigg\vert^2 \, dx.
\end{equation*}
Then 
\begin{equation*}
\begin{split}
I &\leq \delta \sum_{\substack{m_1,m_2 \\ n_1,n_2}}\bigg\vert \int \sinc^2(\delta (x-c))\times \\
&\quad \quad \quad \quad e\left(\frac{x\phi_1(m_1)}{M^{\alpha}} - \frac{x\phi_1(m_2)}{M^{\alpha}} + \frac{x\phi_2(n_1)}{M^{\alpha}} - \frac{x\phi_2(n_2)}{M^{\alpha}}\right) \, dx\bigg\vert \\
\end{split}
\end{equation*}
\begin{equation*}
\begin{split}
&\leq \delta \sum_{\substack{m_1,m_2 \\ n_1,n_2}}\Lambda\left(\delta^{-1}\left(\frac{\phi_1(m_1)}{M^{\alpha}}-\frac{\phi_1(m_2)}{M^{\alpha}} + \frac{\phi_2(n_1)}{M^{\alpha}} - \frac{\phi_2(n_2)}{M^{\alpha}}\right)\right) \\
& \leq \mathcal{N}_{\phi_1,\phi_2}(M,\delta),
\end{split}
\end{equation*}
by Fourier inversion. Also, by positivity we get
\begin{equation*}
\begin{split}
I \geq &\delta \int_{c}^{c+(2\delta)^{-1}}\sinc^2(\delta (x-c)) \bigg\vert\sum_{m = M+1}^{2M} a_m e\left(\frac{x\phi_1(n)}{M^{\alpha}}\right)\bigg\vert^2 \times \\
&\quad \quad \bigg\vert\sum_{n = M+1}^{2M} b_ne\left(\frac{x\phi_2(n)}{M^{\alpha}}\right)\bigg\vert^2 \, dx.
\end{split}
\end{equation*}
Note that in the interval $[c,c+(2\delta)^{-1}]$, we have $\sinc^2(\delta(x-c)) \geq 4/\pi^2$ whence~\eqref{eq:lowerbound} follows by taking $\delta = 2/X$. The proof of~\eqref{eq:upperbound} follows along similar lines.
\end{proof}
For complex numbers $z_1,\ldots,z_N$ set $$\bigg\vert \sum_{1 \leq n \leq N} z_n\bigg\vert ^* = \max_{1 \leq N_1 \leq N_2 \leq N}\bigg \vert \sum_{N_1\leq n \leq N_2}z_n\bigg \vert.$$ 

\begin{lemma}\label{lemma2}
Let $k \geq 1$ be a natural number. We have
\begin{equation*}
\left(\bigg\vert \sum_{1 \leq n \leq N} z_n\bigg\vert ^*\right)^k \leq (1+ \log N)^{k-1}\int_{-1/2}^{1/2}\bigg\vert\sum_{n=1}^Nz_ne(nt)\bigg\vert^k \mathcal{L}(t) \, dt,
\end{equation*}
with $\mathcal{L}(t) = \min\left(N,1/2|t|\right)$ and $\int_{-1/2}^{1/2}\mathcal{L}(t) \, dt = 1+ \log N.$
\end{lemma}
\begin{proof}
See~\cite[Lemma 2]{RS06}.
\end{proof}

\begin{lemma}\label{clever}
Let $0 < Y \leq X$. Let $a_n,b_n$ be complex numbers such that $|a_n|,|b_n| \leq 1$. We have
\begin{equation*}
\begin{split}
&\frac{1}{X}\int_0^X\left(\bigg\vert\sum_{M \leq m \leq 2M}a_me\left(\frac{x\phi_1(m)}{M^{\alpha}}\right)\bigg\vert^*\right)^2\left(\bigg\vert\sum_{M \leq n \leq 2M}b_ne\left(\frac{x\phi_2(n)}{M^{\alpha}}\right)\bigg\vert^*\right)^2\, dx \\
&\leq \frac{\pi^4}{2}(1+\log M)^4\times \\
&\quad \quad \frac{1}{Y}\int_0^Y\left(\bigg\vert\sum_{M \leq m \leq 2M}a_me\left(\frac{x\phi_1(m)}{M^{\alpha}}\right)\bigg\vert\right)^2\left(\bigg\vert\sum_{M \leq n \leq 2M}b_ne\left(\frac{x\phi_2(n)}{M^{\alpha}}\right)\bigg\vert\right)^2\, dx.
\end{split}
\end{equation*}
\end{lemma}
\begin{proof}
By Lemma~\ref{lemma2} we have
\begin{equation*}
\begin{split}
&\frac{1}{X}\int_0^X\left(\bigg\vert\sum_{M \leq m \leq 2M}a_me\left(\frac{x\phi_1(m)}{M^{\alpha}}\right)\bigg\vert^*\right)^2\left(\bigg\vert\sum_{M \leq n \leq 2M}b_ne\left(\frac{x\phi_2(n)}{M^{\alpha}}\right)\bigg\vert^*\right)^2\, dx \\
&\leq (1+\log M)^2 \int_{-1/2}^{1/2}\int_{-1/2}^{1/2}\mathcal{L}(t_1)\mathcal{L}(t_2) \times \\
& \quad \quad \frac{1}{X}\int_0^X\left(\bigg\vert\sum_{M \leq m \leq 2M}a_me(mt_1)e\left(\frac{x\phi_1(m)}{M^{\alpha}}\right)\bigg\vert\right)^2 \times \\ 
&\quad \quad \quad \quad \left(\bigg\vert\sum_{M \leq n \leq 2M}b_ne(nt_2)e\left(\frac{x\phi_2(n)}{M^{\alpha}}\right)\bigg\vert\right)^2\, dx \, dt_1 \, dt_2 \\
&\leq \pi^2/4 (1+\log M)^4 \mathcal{N}_{\phi_1,\phi_2}(M,2/X),
\end{split}
\end{equation*}
by~\eqref{eq:lowerbound}. Since $Y \leq X$, we have 
$$
\mathcal{N}_{\phi_1,\phi_2}(M,2/X) \leq \mathcal{N}_{\phi_1,\phi_2}(M,2/Y).
$$
To complete the proof of the lemma, we invoke~\eqref{eq:upperbound} to bound $\mathcal{N}_{\phi_1,\phi_2}(M,2/Y)$.
\end{proof}
Next, we will apply the van der Corput B--process to the exponential sums $\sum_{n = M+1}^{2M}e\left(\frac{x\phi_i(n)}{M^{\alpha}}\right)$ for $i = 1,2$. 
\begin{lemma}\label{vdc}
Let $2 \leq M < M_1 \leq 2M$ and $M \ll X \ll M^2$. 
Suppose that $X \leq x \leq 2X$. 
Then there exist integers $L$ and $L_1$ which do not depend on $x \in [X,2X]$ nor on $M_1 \in [M+1,\ldots,2M]$ and real numbers $\tau= \tau(x)$, 
with
$$
1 \leq L \leq L_1 \ll L \asymp \frac{X}{M}, \, \, \tau \asymp X
$$
such that
\begin{align*}
\sum_{n=M+1}^{M_1}e\left(\frac{x\phi_i(n)}{M^{\alpha}}\right) 
&\ll \frac{M}{X^{\frac{1}{2}}}\bigg\vert\sum_{L < l \leq L_1}e\left(\tau \left(\frac{l}{L}\right)^{\overline{\alpha}} + l\theta_i\right)\bigg\vert^* + M^{\frac{1}{2}}
\end{align*}
with 
$\overline{\alpha} = \frac{\alpha}{\alpha - 1}$. 
We may take $L = [2^{-|\alpha|-1}|\alpha|XM^{-1}]$ and $L_1 = [2^{2+|\alpha|}|\alpha|XM^{-1}]$ and $\tau(x) = (1-\alpha)\left(\frac{LM}{|\alpha|}\right)^{\overline{\alpha}}x^{\frac{1}{1-\alpha}}.$
\end{lemma}
\begin{proof}
Let 
$f(x) = \frac{x\phi_1(n)}{M^{\alpha}}$ or $\frac{x\phi_2(n)}{M^{\alpha}}$. 
Since $M \ll X \ll M^2$ and $x \in [X,2X]$, we have by~\cite[Theorem 8.16]{IK04} that
$$
\sum_{n=M+1}^{M_1}e(f(n)) = c\sum_{a \leq l \leq b}\frac{e(f(x_l)-lx_l)}{\sqrt{|f''(x_l)|}} + O(M^{1/2}),
$$
where $c = e^{i\pi/4}$ if $f'' > 0$ and $c = e^{-i \pi/4}$ if $f'' < 0$, $a$ and $b$ are such that $f'([M,M_1]) = [a,b]$ and $x_l$ is the unique solution to the equation $f'(x) = l$. Let $f^*(l) = f(x_l)-lx_l$. 
By partial summation, we get
$$
c\sum_{a \leq l \leq b}\frac{e(f^*(l))}{\sqrt{|f''(x_l)|}} \ll \frac{M}{X^{1/2}}\bigg\vert\sum_{a \leq l \leq b}e(f^*(l))\bigg\vert^*.
$$
It is easy to see that 
$$
L \ll |f'(m)| \ll L_1
$$
whenever $x \in [X,2X]$ and $M+1 \leq m \leq M_1$. Therefore, we have
$$
\sum_{n=M+1}^{M_1}e(f(n)) = \frac{M}{X^{1/2}}\bigg\vert\sum_{L \leq l \leq L_1}e(f^*(l))\bigg\vert^* + O(M^{1/2}).
$$
All that is left is to compute $f^*(l)$. Observe that if 
$$
g(m) = A(m+\theta)^{\alpha}, \, \, \text{then}\,\,\, g^{*}(l) = (1-\alpha)A^{\frac{1}{1-\alpha}}\left(\frac{l}{\alpha}\right)^{\overline{\alpha}} + l\theta.$$
To see this, we begin by noting that $g'(x) = A\alpha (x+\theta)^{\alpha-1}$. 
Thus $g'(x_l) = l$ if $\alpha A (x_l+\theta)^{\alpha - 1} = l$. 
In other words, $(x_l+\theta) = \left(\frac{l}{A\alpha}\right)^{\frac{1}{\alpha-1}},$ whence 
\begin{align*}
g^*(l) &= A(x_l+\theta)^{\alpha} - l(x_l+\theta) + l\theta \\
&= l(x_l+\theta)\left(\frac{1}{\alpha} - 1\right) + l\theta \\
&= A^{\frac{1}{1-\alpha}}(1-\alpha)\left(\frac{l}{\alpha}\right)^{\overline{\alpha}} + l\theta.
\end{align*}
Applying this with $A = x/M^{\alpha}$ completes the proof of the lemma.
\end{proof}

\subsection{A recurrence relation}
Let $\alpha \neq 0,1$ be a real number. For a real number $0 < \kappa < 1$ and integer $M \geq 2$, define
\begin{equation}\label{eq:itma}
\begin{split}
I(\kappa,M,\alpha) = \int_0^{M^{2-\kappa}}&
\bigg\vert\sum_{M \leq m \leq 2M}e\left(x\left(\frac{\phi_1(m)}{M}\right)^{\alpha}\right)\bigg\vert^2 \times \\
&\quad \quad \quad \bigg\vert\sum_{M \leq n \leq 2M}e\left(x\left(\frac{\phi_2(n)}{M}\right)^{\alpha}\right)\bigg\vert^2 \, dx,
\end{split}
\end{equation}
with an analogous definition for $I(\kappa,M,\overline{\alpha}).$ As in~\cite[\S 4]{RS06} we make the following
\begin{definition} We say that hypothesis $\mathcal{H}(\alpha,\kappa)$ is satisfied if 
$$
I(\kappa,M,\alpha) \ll_{\ve} M^{4+\ve} \,\,\,\, \text{and} \,\,\,\, I(\kappa,M,\overline{\alpha}) \ll_{\ve} M^{4+\ve}.
$$
\end{definition}
We begin with the following result.
\begin{lemma}
The hypothesis $\mathcal{H}(\alpha,\kappa)$ is true with $\kappa = 1/2$.
\end{lemma}
\begin{proof}
Our task is to show that the integral in~\eqref{eq:itma} with $\kappa = 1/2$ is $O(M^{4+\ve}).$ Let $X_0 = \eta M$, where $\eta > 0$ is a small real number. We write this integral as
$$
I(1/2,M,\alpha) = \int_{0}^1 + \int_1^{X_0} + \int_{X_0}^{M^{3/2}} = I_1 + I_2 + I_3,
$$
say. We trivially have $I_1 \ll M^4$. For $1 \leq x \leq X_0$, by~\cite[Lemma 8.8]{IK04} we obtain the bound 
$$
\sum_{M \leq m \leq 2M}e\left(x\left(\frac{\phi_1(m)}{M}\right)^{\alpha}\right) \ll M/x,
$$ 
whence $I_2 \ll M^{4+\ve}.$ Therefore, it only remains to bound 
$$
I_3 = \int_{X_0}^{M^{3/2}}\bigg\vert\sum_{M \leq m \leq 2M}e\left(x\left(\frac{\phi_1(m)}{M^{\alpha}}\right)\right)\bigg\vert^2
\bigg\vert\sum_{M \leq m \leq 2M}e\left(x\left(\frac{\phi_2(m)}{M^{\alpha}}\right)\right)\bigg\vert^2 \, dx.
$$
We have
\begin{align*}
I_3 &\leq \left(\max_{X_0 \leq x \leq M^{3/2}}\bigg\vert\sum_{M \leq m \leq 2M}e\left(x\left(\frac{\phi_1(m)}{M^{\alpha}}\right)\right)\bigg\vert^2\right)\times \\
&\quad \quad \int_0^{M^{3/2}}\bigg\vert\sum_{M \leq m \leq 2M}e\left(x\left(\frac{\phi_2(m)}{M^{\alpha}}\right)\right)\bigg\vert^2 \, dx.
\end{align*}
Using van der Corput's inequality~\cite[Corollary 8.13]{IK04}, we have
$$
\sum_{M \leq m \leq 2M}e\left(x\left(\frac{\phi_1(m)}{M^{\alpha}}\right)\right) \ll (xM^{-2})^{1/2}M + (xM^{-2})^{-1/2},
$$
which is bounded by $M^{3/4}$ since $X_0 \leq x \leq M^{3/2}$. We have by Lemma~\ref{ulbound} that the second factor is 
$$
\leq M^{3/2}\#\left\{m_1,m_2 \sim M : |\phi_2(m_1)-\phi_2(m_2)| \leq M^{-3/2}\right\} \ll M^{5/2}.
$$
This completes the proof of the lemma.
\end{proof}
Next, we prove the following recurrence relation, which is the key step in the proof of Theorem~\ref{prop1}. 
\begin{lemma}\label{mainlemma}
Suppose that Let $0 < \kappa \leq 1/2$. Suppose that the hypothesis $\mathcal{H}(\alpha, \kappa)$ is satisfied. Then so is the hypothesis $\mathcal{H}\left(\alpha,\frac{\kappa}{1+\kappa}\right).$
\end{lemma}
\begin{proof}
Let $\kappa_1 = \frac{\kappa}{1+\kappa}$. It will suffice to show that $I(\kappa_1,M,\alpha) \ll M^{4+\ve}$.
Consider the integral
\begin{equation*}
J(X) = \int_X^{2X}\bigg\vert\sum_{n = M+1}^{2M}e\left(\frac{x\phi_1(n)}{M^{\alpha}}\right)\bigg\vert^2\bigg\vert\sum_{n = M+1}^{2M}e\left(\frac{x\phi_2(n)}{M^{\alpha}}\right)\bigg\vert^2\, dx.
\end{equation*}
By a standard dyadic argument, we have
$$
\int_0^{M^{2-\kappa_1}} \ll \int_0^{M^{2-\kappa}} + (\log M)\max_{M^{2-\kappa} \ll X \ll M^{2-\kappa_1}} \int_X^{2X}.
$$ 
Thus by the recurrence hypothesis, we have 
$$
I(\kappa_1,M,\alpha) \ll M^{4+\ve} + (\log M)\max_{M^{2-\kappa} \ll X \ll M^{2-\kappa_1}} J(X).
$$ 
Applying Lemma~\ref{vdc} and then making the substitution $y = \tau(x)$ in the ensuing integral, we get that 
\begin{equation*}
\begin{split}
J(X) &\ll \frac{M^4}{X^2}\int_0^{X_1} \left(\bigg\vert \sum_{L \leq l \leq L_1}e\left(y\left(\frac{l}{L}\right)^{\overline{\alpha}} + l\theta_1\right)\bigg\vert^*\right)^2\times \\
& \quad \quad \left(\bigg\vert \sum_{L \leq l \leq L_1}e\left(y\left(\frac{l}{L}\right)^{\overline{\alpha}} + l\theta_2\right)\bigg\vert^*\right)^2\, dx + M^4,
\end{split}
\end{equation*}
where $X_1 \asymp X.$ Applying Lemma~\ref{clever} with 
$$
a_l = \begin{cases} e(l\theta_1) &\mbox{ if $L \leq l \leq L_1$} 
\\ 0 &\mbox{ otherwise,}\end{cases}
$$ 
and similarly for $b_l,$ we get for any $Y \leq X_1$ that
\begin{equation*}
\begin{split}
J(X) &\ll \frac{M^4(\log M)^4}{XY}\int_0^Y \bigg\vert \sum_{L \leq l \leq 2L}e\left(y\left(\frac{l}{L}\right)^{\overline{\alpha}} \right)\bigg\vert^4 \, dx + M^4.
\end{split}
\end{equation*}
Choose $Y = L^{2-\kappa} \asymp (X/M)^{2-\kappa}$. Then we have by~\cite[Lemma 7]{RS06} that
\begin{equation*}
\int_0^Y \bigg\vert \sum_{L \leq l \leq 2L}e\left(y\left(\frac{l}{L}\right)^{\overline{\alpha}} \right)\bigg\vert^4 \, dy \ll L^{4+\ve}.
\end{equation*}
Therefore, we get
$$
J(X) \ll M^{2-\kappa+\ve}X^{1+\kappa} + M^4.
$$
Since $X \ll M^{2-\kappa_1}$, we get that 
$
J(X) \ll M^{4+\ve} 
$
and the lemma follows.
\end{proof}

\begin{lemma}\label{lemmam2}
Let $\alpha \neq 0,1$. Then for any integer $M\geq 2$ and any real number $\ve >0$ we have
$$\int_0^{M^2}\left(\bigg\vert\sum_{M < m \leq 2M}e\left(\frac{x\phi_1(m)}{M^{\alpha}}\right)\bigg\vert^*\right)^2 
\left(\bigg\vert\sum_{M < m \leq 2M}e\left(\frac{x\phi_2(m)}{M^{\alpha}}\right)\bigg\vert^*\right)^2 dx \ll M^{4+\ve}.$$
\end{lemma}
\begin{proof}
Set 
$$
I_M = \int_0^{M^2}\left(\bigg\vert\sum_{M < m \leq 2M}e\left(\frac{x\phi_1(m)}{M^{\alpha}}\right)\bigg\vert^*\right)^2 
\left(\bigg\vert\sum_{M < m \leq 2M}e\left(\frac{x\phi_2(m)}{M^{\alpha}}\right)\bigg\vert^*\right)^2 \, dx
$$
Fix $\ve_0 > 0.$ Applying Lemma~\ref{clever} with $X= M^2$ and $Y = M^{2-\ve_0}$ we get
\begin{align*}
I_M \ll (\log M)^4M^{\ve_0}&\int_0^{M^{2-\ve_0}}\bigg\vert\sum_{M < m \leq 2M}e\left(\frac{x\phi_1(m)}{M^{\alpha}}\right)\bigg\vert^2\times \\
&\quad \quad \bigg\vert\sum_{M < m \leq 2M}e\left(\frac{x\phi_2(m)}{M^{\alpha}}\right)\bigg\vert^2\, dx.
\end{align*}
Set $\kappa_1 = 1/2$ and $\theta_l = \frac{\theta_{l-1}}{\theta_{l-1}+1} = \frac{1}{l+1}.$ We will choose $l$ such that $\frac{1}{l+1} < \ve_0$. Applying Lemma~\ref{mainlemma} $l$ times we find that
$
I_M \ll M^{4+\ve_0+\ve}
$
for any $\ve > 0$. This completes the proof.
\end{proof}

\subsection{Proof of Theorem~\ref{prop1}}
We must now show that 
$$
\mathcal{N}(M,\delta) \ll M^{2+\ve} + \delta M^{4+\ve}.
$$
Suppose first that $\delta = M^{-2}$. Then by Lemma~\ref{ulbound} we get
\begin{align*}
\mathcal{N}(M,M^{-2}) &\leq \pi^2M^{-2}\int_0^{M^2/2}\left(\bigg\vert\sum_{M < m \leq 2M}e\left(\frac{x\phi_1(m)}{M^{\alpha}}\right)\bigg\vert^*\right)^2 \times \\
&\quad \quad \quad\quad  \left(\bigg\vert\sum_{M < m \leq 2M}e\left(\frac{x\phi_2(m)}{M^{\alpha}}\right)\bigg\vert^*\right)^2 \, dx.
\end{align*}
Therefore we are through in this case by appealing to Lemma~\ref{lemmam2}.

If $\delta \leq M^{-2}$, then $\mathcal{N}(M,\delta) \leq \mathcal{N}(M,M^{-2}) \ll M^{2+\ve}$. Finally, if $\delta > M^{-2}$ we apply Lemmas~\ref{ulbound} and~\ref{lemmam2} once again to get
\begin{align*}
\mathcal{N}(M,\delta) &\ll \delta \int_0^{\delta^{-1}}\left(\bigg\vert\sum_{M < m \leq 2M}e\left(\frac{x\phi_1(m)}{M^{\alpha}}\right)\bigg\vert^*\right)^2 \times \\
&\quad \quad \quad \quad \quad \quad \left(\bigg\vert\sum_{M < m \leq 2M}e\left(\frac{x\phi_2(m)}{M^{\alpha}}\right)\bigg\vert^*\right)^2 \, dx.
\end{align*}
Thus we get
\begin{align*}
\mathcal{N}(M,\delta) &\ll \delta \int_0^{M^2}\left(\bigg\vert\sum_{M < m \leq 2M}e\left(\frac{x\phi_1(m)}{M^{\alpha}}\right)\bigg\vert^*\right)^2 \times \\
&\quad \quad \quad \quad \quad \quad \left(\bigg\vert\sum_{M < m \leq 2M}e\left(\frac{x\phi_2(m)}{M^{\alpha}}\right)\bigg\vert^*\right)^2 \, dx \\
&\ll \delta M^{4+\ve}.
\end{align*} 
This completes the proof of the theorem.

\end{document}